\documentclass{amsart}

\usepackage{latexsym}
\usepackage{amsmath}
\usepackage{amssymb}

\usepackage{multicol}

\usepackage[english]{babel}

\theoremstyle{plain}

\newtheorem{thm}{Theorem}[section]

\newtheorem{cor}[thm]{Corollary}
\newtheorem{lemma}[thm]{Lemma}

\theoremstyle{definition}
\newtheorem{defn}[thm]{Definition}
\newtheorem{rem}[thm]{Remark}
\numberwithin{equation}{section}

\newcommand{\sph}{S}
\newcommand{\cn}{\mathbb{C}^n}
\newcommand{\cpn}{\mathbb{CP}^{n-1}}

\newcommand{\si}{\sigma}
\newcommand{\wsn}{\widehat{\sigma}_n}

\newcommand{\Dbb}{\mathbb{D}}
\newcommand{\Tbb}{\mathbb{T}}
\newcommand{\Nbb}{\mathbb{N}}

\newcommand{\zz}{\mathbb{Z}_+^2}

\newcommand{\spn}{S_n}

\newcommand{\za}{\zeta}
\newcommand{\de}{\delta}
\newcommand{\er}{\varepsilon}
\newcommand{\la}{\lambda}
\newcommand{\al}{\alpha}
\newcommand{\riz}{\Pi}
\newcommand{\EE}{\mathcal{E}\!}
\newcommand{\Hau}{\mathcal{H}\!}
\newcommand{\MM}{M_+}
\newcommand{\tn}{\mathbb{T}^n}

\newcommand{\spec}{\textrm{spec}}

\begin{document}

\date{}

\author{Evgueni Doubtsov}
\address{St.~Petersburg Department
of Steklov Mathematical Institute, Fontanka 27, St.~Petersburg 191023, Russia}
\email{dubtsov@pdmi.ras.ru}

\title[Dimensions of Riesz products]{Dimensions of Riesz products and pluriharmonic measures}

\begin{abstract}
We estimate the energy and Hausdorff dimensions of the Riesz products
on the unit sphere of $\mathbb{C}^n$, $n\ge 2$.
Also, we obtain similar results for the pluriharmonic measures on the torus.  
\end{abstract}



\maketitle

\section{Introduction}\label{s_int}

The present paper is motivated by the following general question:
how does the spectrum of a measure $\mu$ affect the size of support of $\mu$?
We primarily 
consider Riesz product measures and related objects 
on the unit sphere $S = \spn = \{\za\in\cn : |\za| = 1\}$, $n \ge 2$,
therefore, we start by introducing the spaces $H(p,q)$, $(p,q)\in\zz$.

\subsection{Complex spherical harmonics}
Let $\mathcal{U}(n)$ denote the group of unitary operators on the Hilbert space $\cn$, $n\ge 2$.
Observe that $\spn = \mathcal{U}(n)/\mathcal{U}(n-1)$, hence, $\spn$ is a homogeneous space.
General constructions of abstract harmonic analysis are explicitly implemented
on $\spn$ in terms of the spaces $H(p,q)$, $(p,q)\in\zz$. 

\begin{defn}
Fix a dimension $n$, $n\ge 2$.
Let $H(p, q)= H(p,q; n)$ denote the space of all homogeneous harmonic polynomials
of bidegree $(p, q) \in \zz$.
By definition, this means that the polynomials under consideration have degree
$p$ in $z_1, z_2, \dots, z_n$, degree $q$ in 
$\overline{z}_1, \overline{z}_2, \dots, \overline{z}_n$,
and have total degree $p+q$.

For the restriction of $H(p, q)$ on $S$, one uses the same symbol.
The elements of $H(p, q)$ are often called \textit{complex} spherical harmonics.
\end{defn}

\subsection{Riesz products on the sphere}
Zygmund's dichotomy \cite{Zy} suggests that the homogeneous holomorphic polynomials introduced by
Ryll and Wojtaszczyk~\cite{RW83} could provide examples
of singular Riesz-type products on the sphere.
We say that $\{R_j\}_{j=1}^\infty$
 is a Ryll--Wojtaszczyk sequence with a constant $\delta \in (0, 1)$ if
\begin{itemize}
\item $R_j \in H(j, 0)$, that is, $R_j$ is a homogeneous holomorphic polynomial of degree~$j$,
\item $\|R_j\|_{L^\infty(\sph)} = 1$,
\item $\|R_j\|_{L^2(S)} \ge \delta$ for all $j =1,2,\dots$.
\end{itemize}

\begin{defn}\label{d_Rpair}
Let $R = \{R_j\}_{j=1}^\infty$ be a Ryll--Wojtaszczyk sequence,
$J = \{j_k\}_{k=1}^\infty\subset \Nbb$, $j_{k+1}/j_k \ge 3$,
and $a = \{a_k\}_{k=1}^\infty \subset \Dbb$.
Then $(R,J,a)$ is called a Riesz triple.
\end{defn}

Each Riesz triple $(R, J, a)$ generates a (standard) Riesz product.
Namely, the Riesz product
$\riz(R, J, a)$ is defined by the formal equality
\[
\riz(R, J, a) = \prod_{k=1}^\infty
\left(
\frac{\overline{a}_k \overline{R}_{j_k}}{2}
+ 1 +
\frac{a_k R_{j_k}}{2}
\right),
\]
where the product converges in the weak*-sense (see \cite{DouAIF}, \cite{D24singRiz} for details).

There exists a fairly rich set of singular Riesz products on the sphere.
Indeed, if $(R, J, a)$ is a Riesz pair with $a\notin\ell^2$, then Corollary~1 from \cite{D24singRiz} provides
a sequence $U=\{U_j\}_{j=1}^\infty$, $U_j \in \mathcal{U}(n)$, such that $\riz(R\circ U, J, a)$
is singular with respect to Lebesgue measure on $\spn$.

For a singular Riesz product, it is natural to ask about its dimension.
In the present paper, we estimate the energy and Hausdorff dimensions of the Riesz products
on the sphere.
Let $\dim_\Hau \riz$ denote the Hausdorff dimension of a measure $\riz$.
We obtain, in particular, the following result.

 \begin{thm}\label{t_Haus}
Let $(R, J, a)$ be a Riesz triple.
Then 
\[
\dim_\Hau \riz (R, J, a) \ge 2n -1 - \limsup_{k\to \infty}
\left(
\frac{1}{2\log j_k}
\sum_{\ell=1}^{k-1}
|a_\ell|^2 
\right).
\]
\end{thm}

\subsection{Organization of the paper}
Auxiliary results, including definitions and basic facts of harmonic analysis on $\spn$,
are collected in Section~\ref{s_aux}.
Estimates of the energy and Hausdorff dimensions of the Riesz products
on the unit sphere are obtained in Section~\ref{s_dim_riz}.
A related problem about 
Hausdorff dimensions of the pluriharmonic measures on the torus is discussed in Section~\ref{s_plh}.

\subsection{Notation} For $F, G>0$, we write $F\lesssim G$ provided that $F\le C G$ for a constant $C>0$.
If $F\lesssim G$ and $G\lesssim F$, then we write $F\approx G$.

\section{Auxiliary results}\label{s_aux}
\subsection{Basics of harmonic analysis on $\spn$}\label{ss_basic}
Let $\si=\si_n$ denote the normalized Lebesgue measure on the unit sphere $\spn$. Observe that
\begin{equation}\label{e_orth}
L^2(\si) = \operatornamewithlimits{\oplus}_{(p,q)\in \zz} H(p, q).
\end{equation}
Specific aspects of the harmonic analysis on $S$ are illustrated by the following
multiplication rule for the spaces $H(p, q)$: if $f \in H(p, q)$ and
$g \in H(r, s)$, then
\[ fg\in
\sum_{\ell=0}^L
H(p + r -\ell, q + s -\ell),
\]
where $L = \min(p, s) + \min(q, r)$.
See \cite[Chapter~12]{Ru80} for the proofs of the above facts and
for further results.

Let $M(\spn)$ denote the space of complex Borel measures on the unit sphere $\spn$.
Let $K_{p,q}(z, \za)$ denote the reproducing kernel for the Hilbert space
$H(p, q) \subset L^2(\sph)$. The polynomial
\[
\mu_{p,q}(z) = \int_S
K_{p,q} (z, \za)\, d\mu(\za), \quad z\in S,
\]
is called the $H(p, q)$-projection of $\mu\in M(S)$.
For $\mu\in M(S)$, the spectrum $\spec(\mu)$ in terms of $H(p,q)$ is defined by
\[
\spec(\mu) =
\left\{
(p, q) \in \zz:\, \mu_{p,q} \neq \mathbf{0}
\right\}.
\]


\subsection{Disintegration into slice-products}
Let $\Tbb=\sph_1$ denote the unit circle.
Given a $\xi\in \spn$,
the slice-product
\[
\riz_\xi(R, J, a)(\la) := \riz(R(\la\xi), J, a),\quad \la\in \Tbb,
\]
is the classical Riesz product
\begin{equation}\label{e_slice}
\prod_{k=1}^\infty \left(
\frac{\overline{a}_k \overline{R}_{j_k} (\xi)\overline{\la}^{j_k}}{2} + 1 +
\frac{{a}_k R_{j_k}(\xi) {\la}^{j_k}}{2}
\right),
\quad \la \in\Tbb.
\end{equation}
In particular, $\riz_\xi(R, J, a)$ is a correctly defined probability measure.

Let $\cpn$ denote
the complex projective space of dimension $n-1$,
that is, the collection of all one-dimensional
linear subspaces of $\cn$.
Let $\pi=\pi_n$ denote the canonical projection from
$\spn$ onto $\cpn$.
Observe that $\riz_{\lambda\xi}(w) = \riz_\xi (\la w)$ for $\xi\in\spn$, $\lambda, w\in\Tbb$.
Therefore, the probability slice-measure $\riz_\za$ is correctly defined for $\za\in\cpn$ as
an element of $\MM(\spn)$ supported by the unit circle $\pi^{-1}(\za) \subset \spn$.

Let $\wsn$ denote the unique probability measure on $\cpn$
 invariant with respect to all unitary transformations of $\cn$.
The following lemma shows that $\riz(R, J, a)$ is the integral of its slices.

\begin{lemma}[{\cite[Lemma~1]{D24singRiz}}]\label{l_disint}
Let $(R, J, a)$ be a Riesz triple on the sphere $S_n$, $n\ge 2$,
and let $\riz(R, J, a)$ denote the corresponding Riesz product. Then
\[
 \riz(R, J, a) = \int_{\cpn} \riz_\za(R, J, a)\, d\wsn(\za)
\]
in the following weak sense:
\[
 \int_{S_n} f\, d \riz(R, J, a) = \int_{\cpn} \int_{\spn} f\, d\riz_\za(R, J, a)\, d\wsn(\za)
\]
for all $f\in C(S_n)$.
\end{lemma}

\section{Dimensions of Riesz products}\label{s_dim_riz}

\subsection{Energy dimension}
Let $\MM(\spn)$ denote the set of all finite positive Borel measures on $\spn$.
Given a measure $\mu\in \MM(S)$, its $t$-energy $I_t(\mu)$, $t>0$, is defined as
\[
I_t(\mu) := \int_S \int_S \frac{d\mu(x)d\mu(y)}{|x-y|^t}.
\]
We say that $d$ is the energy dimension of $\mu$
and we write $\dim_\EE \mu =d$ if
\[
d = \sup \left\{t : I_t(\mu) < \infty\right\}.
\]

To estimate the energy dimension of $\riz(R, J, a)$,
we need the following theorem.

\begin{thm}[{\cite[Theorem 3.1]{HR03Ark}}]\label{t_RH_Esph}
For each $t$, $0<t<2n-1$, there are constants $C_1, C_2 > 0$ such that
\[
C_1 I_t(\mu) \le \|\mu_{0,0}\|_2^2  + \sum_{j=1}^\infty j^{t-2n+1} \sum_{p+q=j} \|\mu_{p,q}\|_2^2
\le C_2 I_t(\mu)
\]
for all $\mu\in\MM(\spn)$ with
$n \ge 2$.
\end{thm}

In \cite{HR03PAMS}, an analog of Theorem~\ref{t_Riesz_Edim} for the unit circle is used
to compute the energy dimension of a classical Riesz product on $\Tbb$.
In the following result, we apply Theorem~\ref{t_RH_Esph}
to estimate from below the energy dimension of a Riesz product on the sphere.
In fact, modulo technical details, the principal argument below is similar to that from the proof 
of Theorem~3.1 from \cite{HR03PAMS}.

\begin{thm}\label{t_Riesz_Edim}
Let $(R, J, a)$ be a Riesz triple.
Then
\[
\dim_\EE \riz(R, J, a) \ge 2n-1-\al_0,
\]
where
\begin{equation}\label{e_a0}
\al_0 = \max
\left[0,\
\limsup_{k\to \infty}
\left(
\frac{
\log \frac{|a_k|^2}{2} +
\sum_{\ell=1}^{k-1}
\log \left(1 + \frac{|a_\ell|^2}{2} \right)
}
{
\log j_k}
\right)
\right].
\end{equation}
In particular,
\[
\dim_\EE \riz (R, J, a) \ge 2n -1 - \limsup_{k\to \infty}
\left(
\frac{1}{2\log j_k}
\sum_{\ell=1}^{k-1}
|a_\ell|^2 
\right).
\]
\end{thm}
\begin{proof}
 Let
$\riz:= \riz(R, J, a)$.
For $k=1,2,\dots$, let
\[
\Gamma_k := \left\{\pm j_k +
\sum_{\ell=1}^{k-1} \er_\ell j_\ell: \er_\ell = 0, \pm 1\right\}.
\]
Suppose that $\gamma\in \Gamma_k$. Without loss of generality, assume that $\gamma>0$.
Then there exists exactly one representation
\[
\gamma = j_k + \sum_{\ell=1}^{k-1} \er_\ell j_\ell.
\]
By the multiplication rule for the spherical harmonics (see Section~\ref{ss_basic}),
\[
\sum_{p-q=\gamma} \riz_{p,q} = R_{j_k} \prod_{\ell=1}^{k-1} R_{j_\ell}(\er_\ell),
\]
where
\[
\begin{split}
  R_{j_\ell}(\er_\ell) &= R_{j_\ell}, \quad \er_\ell=1,\\
   R_{j_\ell}(\er_\ell) &= 1, \quad \quad \er_\ell=0, \\
   R_{j_\ell}(\er_\ell) &= \overline{R}_{j_\ell}, \quad \er_\ell=-1.
\end{split}
\]
Applying \eqref{e_orth}, we obtain
\[
\sum_{p-q=\gamma}\|\riz_{p,q}\|_2^2 
= \left\| R_{j_k} \prod_{\ell=1}^{k-1} R_{j_\ell}(\er_\ell)\right\|^2_2
\le
\left(
\frac{|a_k|}{2}\prod_{\ell:\, \er_\ell\neq 0}\frac{|a_\ell|}{2}
\right)^2,
\]
where the empty product is equal to one.

Now, fix an $\al> \al_0$. Since $p+q \ge p-q$, we have
\begin{equation}\label{e_gamma1}
\sum_{\gamma\in \Gamma_k}
|p+q|^{-\alpha} \sum_{p-q=\gamma} \|\riz_{p,q}\|_2^2 \le
 \sum_{\gamma\in \Gamma_k}
|\gamma|^{-\alpha} \sum_{p-q=\gamma} \|\riz_{p,q}\|_2^2.
\end{equation}
If $\gamma \in \Gamma_k$, then $|\gamma| \approx j_k$. Thus,
\begin{equation}\label{e_gamma2}
\sum_{\gamma\in \Gamma_k}
|\gamma|^{-\alpha} \sum_{p-q=\gamma} \|\riz_{p,q}\|_2^2 \lesssim j_k^{-\alpha}\frac{|a_k|^2}{2}
\prod_{\ell=1}^{k-1}
\left(1 +\frac{|a_\ell|^2}{2} \right).
\end{equation}
Combining \eqref{e_gamma1} and \eqref{e_gamma2}, we obtain
\begin{equation}\label{e_energy_upper}
I_{2n-1-\alpha}(\Pi) \lesssim 1+ j_k^{-\alpha}\frac{|a_k|^2}{2}
\prod_{\ell=1}^{k-1}
\left(1 +\frac{|a_\ell|^2}{2} \right)
\end{equation}
by Theorem~\ref{t_RH_Esph}.

We have $\log j_k \ge (k-1)\log 3$, thus,
selecting an $A$, $0<A<1$, sufficiently close to~$1$,
we deduce from the inequality $\al> \al_0$ that
\[
\al\ge
\frac{
\log \frac{|a_k|^2}{2} +
\sum_{\ell=1}^{k-1}
\log \left(1 + \frac{|a_\ell|^2}{2} \right)
+ k |\log A|}
{
\log j_k}
\]
or, equivalently,
\[
j_k^{-\al}
\frac{|a_k|^2}{2}
\prod_{\ell=1}^{k-1}
\left(
1 +
\frac{|a_\ell|^2}{2}
\right)
\le A^k.
\]
Combining the above property and \eqref{e_energy_upper}, we conclude that $I_{2n-1-\al}(\riz) < \infty$.
Since $\al>\al_0$ is arbitrary, we conclude that $\dim_\EE \riz \ge 2n-1-\al_0$, as required.
\end{proof}

\subsection{Hausdorff dimension}

For a Borel set $E$, let $\dim_\Hau E$ denote its Hausdorff dimension.
The Hausdorff dimension of a measure $\mu\in\MM(\spn)$ is defined as
\[
\dim_\Hau \mu = \inf \{\dim_\Hau E : E\ \textrm{is a Borel set with}\ \mu(E) > 0\}.
\]
For properties of the Hausdorff dimension of a measure see \cite[Chapter~10]{Fa97}.

Given a Riesz product $\riz=\riz(R, J, a)$, direct inspection shows that there exists a sufficiently small $\er>0$ such that
$\riz_{p,q}=\mathbf{0}$ provided that 
\[
\left|\frac{p}{q} -1\right|<\er,\quad (p,q)\neq (0,0).
\]
Therefore, Theorem~1.1 from \cite{AW22} applies to $\riz$ and guarantees that 
$\dim_\Hau \riz \ge 2n-2$.
In fact, applying Theorem~\ref{t_Riesz_Edim}, we obtain a more precise estimate.

\subsubsection{Application of Theorem~\ref{t_Riesz_Edim}}
If $I_t(\mu) <\infty$, then $\dim_\Hau \mu > t$ (cf.\ \cite[Section~4.3]{Fa90}); thus, the Hausdorff dimension of a
measure is always at least the energy dimension.

\begin{cor}\label{c_Haus}
Let $(R, J, a)$ be a Riesz triple.
Then 
\[
\dim_\Hau \riz(R, J, a) \ge 2n-1 - \alpha_0,
\]
where $\al_0$ is defined by \eqref{e_a0}.
\end{cor}
\begin{proof}
Since $\dim_\Hau \riz \ge \dim_{\EE} \riz$, it suffices to apply Theorem~\ref{t_Riesz_Edim}.
 \end{proof}
 
\begin{rem}
 Clearly, Theorem~\ref{t_Haus} is a particular case of the above corollary.
 \end{rem}
 
\subsubsection{Reduction to slice-products}
An alternative approach to Corollary~\ref{c_Haus} is to consider the slice-products 
of $\riz(R, J, a)$.
We need the following theorem which is a consequence of \cite[Sections~2.10.2 and 2.10.17]{Fe69}.

\begin{thm}\label{t_federer}
Let $K\subset \cpn \times \Tbb$ be a compact set and $K_\za = \{w\in\Tbb: (\za, w) \in K\}$.
Assume that $\dim_\Hau K_\za > \beta$ for $\za\in X\subset\cpn$.
If $\wsn(X)>0$, then 
\[
\dim_\Hau K \ge 2n-2+\beta.
\]
\end{thm}

Fix an $\al> \al_0$. It suffices to prove that
\begin{equation}\label{e_ge_al}
\dim_{\Hau} \riz(R, J, a) \ge 2n -1 -\al.
\end{equation}
Assume that the above estimate does not hold. Then there exists a compact $E\subset\spn$
such that $\riz(E)>0$ and 
\[
\dim_{\Hau} E < 2n -1 - \al.
\]
We identify $\spn$ and $\cpn \times \Tbb$, and we consider the sets 
$E_\za = \{w\in\Tbb: (\za, w) \in E\}$.
We claim that
\begin{equation}\label{e_le_1al} 
\dim_\Hau E_\za \le 1-\alpha
\quad\textrm{
for $\wsn$-almost all $\za\in\cpn$.}
\end{equation}
Indeed, if this is not the case, then $\dim_\Hau E_\za > 1-\alpha$
for $\za\in X$, where $\wsn(X)>0$.
Applying Theorem~\ref{t_federer}, we obtain $\dim_{\Hau} E \ge 2n -1 - \al$,
a contradiction. Thus, \eqref{e_le_1al} holds.

Since $\riz(E)>0$, Lemma~\ref{l_disint} guarantees that $\riz_\za(E_\za)>0$
for $\za\in Y_0$ with $\wsn(Y_0)>0$.
Combining this fact and \eqref{e_le_1al}, we obtain 
\begin{equation}\label{e_le_slice_1al}
\dim_\Hau \riz_\za \le 1-\al
\quad\textrm{
for $\za\in Y$ with $\wsn(Y)>0$}.
\end{equation}
Now, recall that $\riz_\za(R, J, a)$ is the classical Riesz product defined by \eqref{e_slice}.
Therefore,
$\dim_\Hau \riz_\za(R, J, a) \ge 1 - \alpha_0 > 1- \alpha$ by \cite[Theorem~3.1]{HR03PAMS}.
This contradicts \eqref{e_le_slice_1al}. Thus, \eqref{e_ge_al} holds true
for any $\al> \al_0$, as required.

In other words, consideration of slice-products allows to recover Corollary~\ref{c_Haus}
from \cite[Theorem~3.1]{HR03PAMS}, the corresponding result for the classical Riesz products.

\section{Hausdorff dimensions of pluriharmonic measures}\label{s_plh}
This section is motivated by analogs of Lemma~\ref{l_disint} and their applications
in the setting of pluriharmonic measures.

\subsection{Pluriharmonic measures on the unit sphere}
A measure $\mu\in M(\spn)$ is called pluriharmonic if 
$\spec(\mu) \subset \{(p,q)\in \zz: pq=0\}$.
An analog of Lemma~\ref{l_disint} is known for the pluriharmonic measures on $\spn$.
This fact implies that $\dim_\Hau \mu \ge 2n-2$ for any pluriharmonic measure $\mu\in M(\spn)$,
see \cite{Aab85, AW22} for further details.
However, to the best of the author's knowledge, the sharpness of this estimate is an open problem.

\subsection{Pluriharmonic measures on the torus}
A measure $\mu\in M(\tn)$ is called pluriharmonic if
$\hat{\mu}(k_1, \dots, k_n) = 0$ for 
$(k_1, \dots, k_n) \in \mathbb{Z}^n \setminus (\mathbb{Z}_-^n \cup \mathbb{Z}_+^n)$.
It is well known that $\mu\in M(\tn)$ is pluriharmonic if and only if the Poisson integral of $\mu$
is a pluriharmonic function in the polydisk $\Dbb^n$.
Therefore, an analog of Lemma~\ref{l_disint} holds true for every pluriharmonic measure
 $\mu$ on $\tn$.
 See \cite[Proposition~2.1]{AD20} for the case of $\mu$ on the unit sphere $\spn$,
 see also \cite{Cz24ppt}, where 
 pluriharmonic measures on the Shilov boundary of a bounded symmetric domain are considered.
Similar to the case of $\spn$, we conclude that 
\begin{equation}\label{e_hau_n1}
\dim_\Hau \mu \ge n-1;
\end{equation}
cf.\ \cite[Corollary~3.3]{Cz24ppt}.
For $n=2$, estimate \eqref{e_hau_n1} and its sharpness were earlier proved in \cite{Bq23ppt}.
Simple examples show that estimate \eqref{e_hau_n1} is sharp for all $n\ge 2$. 
Indeed, put 
\[
\mu = \de_{1}(\xi_1) \otimes m(\xi_2)\otimes \cdots \otimes m(\xi_n),
\]
where $\de_{1}$ denotes the Dirac measure at $1\in \Tbb$ 
and $m = \si_1$ denotes the normalized Lebesgue measure on $\Tbb$.
Clearly, $\mu$ is a pluriharmonic measure and $\dim_\Hau \mu = n-1$.


\begin{thebibliography}{10}

\bibitem{Aab85}
A.~B. Aleksandrov, \emph{Function theory in the ball}, Current problems in
  mathematics. {F}undamental directions, {V}ol.\ 8, Itogi Nauki i Tekhniki,
  Akad. Nauk SSSR Vsesoyuz. Inst. Nauchn. i Tekhn. Inform., Moscow, 1985,
  pp.~115--190, 274 (Russian); English transl., Encyclopaedia Math.\ Sci.,
  vol.\ 8, Springer--Verlag, Berlin, 1994, pp.\ 107--178.

\bibitem{AD20}
A.~B. Aleksandrov and E.~Doubtsov, \emph{Clark measures on the complex sphere},
  J. Funct. Anal. \textbf{278} (2020), no.~2, 108314, 30. 

\bibitem{AW22}
R.~Ayoush and M.~Wojciechowski, \emph{Microlocal approach to the {H}ausdorff
  dimension of measures}, Adv. Math. \textbf{395} (2022), Paper No. 108088,
  11pp. 

\bibitem{Bq23ppt}
L.~Bergqvist, \emph{Necessary conditions on the support of {RP}-measures},
  arXiv:2304.03072.

\bibitem{Cz24ppt}
M.~Calzi, \emph{Clark measures on bounded symmetric domains}, arXiv:2403.05429.

\bibitem{D24singRiz}
E.~Doubtsov, \emph{Mutual singularity of {R}iesz products on the unit shpere},
arXiv:2404.02652.

\bibitem{DouAIF}
E.~Doubtsov, \emph{Henkin measures, {R}iesz products and singular sets}, Ann. Inst.
  Fourier (Grenoble) \textbf{48} (1998), no.~3, 699--728. 

\bibitem{Fa90}
K.~Falconer, \emph{Fractal geometry}, John Wiley \& Sons, Ltd., Chichester,
  1990. \MR{1102677}

\bibitem{Fa97}
K.~Falconer, \emph{Techniques in fractal geometry}, John Wiley \& Sons, Ltd.,
  Chichester, 1997. \MR{1449135}

\bibitem{Fe69}
H.~Federer, \emph{Geometric measure theory}, Die Grundlehren der mathematischen
  Wissenschaften, vol. Band 153, Springer-Verlag New York, Inc., New York,
  1969. \MR{257325}

\bibitem{HR03PAMS}
K.~E. Hare and M.~Roginskaya, \emph{A {F}ourier series formula for energy of
  measures with applications to {R}iesz products}, Proc. Amer. Math. Soc.
  \textbf{131} (2003), no.~1, 165--174. \MR{1929036}

\bibitem{HR03Ark}
K.~E. Hare and M.~Roginskaya, \emph{Multipliers of spherical harmonics and energy of measures on the
  sphere}, Ark. Mat. \textbf{41} (2003), no.~2, 281--294. \MR{2011922}

\bibitem{Ru80}
W.~Rudin, \emph{Function theory in the unit ball of {${\bf C}^{n}$}},
  Grundlehren der Mathe\-matischen Wissen\-schaften, vol. 241, Springer-Verlag,
  New York-Berlin, 1980.

\bibitem{RW83}
J.~Ryll and P.~Wojtaszczyk, \emph{On homogeneous polynomials on a complex
  ball}, Trans. Amer. Math. Soc. \textbf{276} (1983), no.~1, 107--116.
  \MR{684495}

\bibitem{Zy}
A.~Zygmund, \emph{Trigonometric series. {V}ol. {I}, {II}}, third ed., Cambridge
  Mathematical Library, Cambridge University Press, Cambridge, 2002.
  \MR{1963498}

\end{thebibliography}

\end{document}